\def\ZZ         {{\mathbb Z}}

\def\CC         {{\mathbb C}}
\def\QQ         {{\mathbb Q}}
\def\PP         {{\mathbb P}}

\def\F          {{\cal F}}
\def\H          {{\cal H}}
\def\L          {{\cal L}}

\def\O          {{\cal O}}

\def\V           {{\cal V}}
\def\W           {{\cal W}}

\def\rk         {{\rm rk}}

\def\dim          {{\rm dim}}

\def\dim        {{\rm dim}}

\def\Ker        {{\rm Ker}}
\def\Coker      {{\rm Coker}}

\def\cal        {\mathcal}

\documentclass{amsart}
\usepackage{amssymb}
\usepackage{verbatim}
\usepackage{eucal}
\usepackage{mathrsfs}
\usepackage{graphicx}
\usepackage{psfrag}

\newtheorem{theorem}{Theorem}[section]
\newtheorem{lemma}[theorem]{Lemma}
\newtheorem{prop}[theorem]{Proposition}
\newtheorem{corollary}[theorem]{Corollary}

\theoremstyle{definition}
\newtheorem{dfn}[theorem]{Definition}

\theoremstyle{remark}
\newtheorem{remark}[theorem]{Remark}

\begin{document}

\title{Cohomology of bundles on homological Hopf manifolds.}

\author{Anatoly Libgober}
\footnote{The author supported by National Science Foundation grant.}
\address{Department of Mathematics\\
University of Illinois\\
Chicago, IL 60607}
\email{libgober@math.uic.edu}

\begin{abstract}
{ We discuss the properties
of complex manifolds having rational homology of 
$S^1 \times S^{2n-1}$ including 
those constructed by Hopf, Kodaira and Brieskorn-van de Ven. 
We extend certain previously known properties of cohomology of bundles 
on such manifolds.As an application we consider degeneration of Hodge-deRham 
spectral sequence in this non Kahler setting.
}  
\end{abstract}

\maketitle

\section{Introduction}
The goal of this paper is to obtain information 
on the cohomology of bundles on some generalizations of Hopf manifolds 
applying methods of \cite{libgober} in the bundle setting. 
Roughly, the cohomology of the bundles $\Omega^p_{\H}(E)$ on a 
topological Hopf manifold $\H$ where 
$E$ is a bundle with a trivial pull back on the universal cover,
controlled by the ``Alexander modules'' naturally identified 
with the appropriate local cohomology associated with 
the universal cover (cf. theorem \ref{cohomologyresults}).
To put things in perspective we start with a review of 
some results and viewpoints on manifolds which are 
natural generalizations of the classical construction of Hopf and Kodaira
(cf. \cite{kodairanas} \cite{kodairaamerjourn}). 
In particular Kodaira had shown that a surface
with universal cover biholomorphic to  $\CC^2-0$ 
is the quotient of the latter by  
a group having as a finite index normal subgroup an infinite cyclic group.
 In higher dimensions fixing the homological type
leaves much more possibilities for the analytic type of the 
universal cover: for example Brieskorn and van de Ven found infinitely many 
examples of manifolds having homeomorphism or even diffeomorphism 
type of $S^1 \times S^{2n-1}$ but having analytically distinct 
universal cover. These universal covers are non singular loci of 
affine hypersurfaces (\ref{pham}) with isolated singularities
and having as their links (possibly exotic) odd dimensional spheres. 
The action of the fundamental group in Brieskorn van de Ven examples
is given by (\ref{homotety}). 
More generally, we consider generalized Hopf manifolds allowing arbitrary 
quotients \footnote{and not only 
corresponding to the actions (\ref{homotety})} 
of the non-singular loci of affine varieties which are  
weighted homogeneous complete intersection with isolated singularity having 
a $\QQ$-sphere as its link. 
If this link is a $\ZZ$-sphere then 
the quotient is homeomorphic to $S^1 \times S^{2n-1}$.
In the case when the universal cover of a primary 
Hopf manifold is 
$\CC^n-0$, Haefliger (cf. \cite{haefl}) 
described actions of the fundamental group 
on the universal cover (extending the case $n=2$ studied by Kodaira)
.
In the case when the universal cover $V-0$ of a homological Hopf manifold
is a complete intersecion (\ref{hammcase}) 
we show that, if the degree of such complete intersection 
is sufficiently large, then the corresponding quotients
are finite quotients of $V-0/\ZZ$ with the action 
of $\ZZ$ given by (\ref{homotety}). 

More precisely we will see the following:

\begin{theorem}\label{universalcoverBvdV} Let 
$\H$ be a homological Hopf manifold having 
as its universal cover the non-singular locus of an affine hypersurface
  $z_1^{a_1}+...+
z_n^{a_r}=0$ or a complete intersection (\ref{hammcase}).
Then $\H$ is a quotient by a finite group 
of the Brieskorn van deVen manifold 
(or respectively 
$(V-0)/{\ZZ}$ where $V$ is a complete intersection (\ref{hammcase}) 
with the action of a generator of $\ZZ$ given by (\ref{homotety})).
\end{theorem} 

This deduced from a result 
on the automorphism group of affine hypersurfaces and 
complete intersections  (\ref{hammcase}) or (\ref{pham})
(cf. \ref{autBvdV}). 

It is an interesting problem 
to find a reasonable classification of homological or topological 
Hopf manifolds with fixed topological, differential or 
almost complex type. The analytic type of the universal 
cover and the action of the fundamental group on it  
provide discrete invariants of ``the components of the moduli space'' 
of complex structures (cf. \ref{propfundgroups}). 
The main point of this note is that regardless 
of specific properties of the analytic structure 
on the universal cover one has an interesting restriction on the cohomology 
of bundles generalizing results of  \cite{mall1}, \cite{mall2} 
(in primary case) and \cite{zhou} (in non primary Hopf manifolds case). 
The main, rather technically looking result, is the following:

\begin{theorem}\label{cohomologyresults}
 Let $H$ be a homological 
Hopf manifold with the universal covering
space $p: V \rightarrow H$ where $V$ is a complement 
to compact set $A$ is a Stein space $\bar V$ and 
$i: V \rightarrow \bar V$. Let $f$ be an element
of the infinite order in $\pi_1(H)$ and let $f_V: V \rightarrow V$.
be the corresponding biholomorphic automorphism of $V$. 
Let $E$ be a bundle on $H$. 
We denote $f^k_V$ the corresponding automorphism of $H^k(V,p^*(E))$.
Assume that the local cohomology group $H^r_A(\bar V,i_*(\pi^*(E)))$
are finite dimensional. 
Then:

a) $\dim H^k(H,E)=\dim \Ker f^k_V+\dim \Coker f^{k-1}_V$.

b) In particular if $H^k(V,p^*(E))=0$ 
 for $a \le k \le b$ then $H^k(H,E)=0$ 
for $a+1 \le k \le b-1$. 
\end{theorem}

This yields the following result for the Hopf and Brieskorn Van de Ven 
manifolds and some homological Hopf manifolds (see next section 
for explanation of terminology):

\begin{theorem}\label{vanishing} Assume $n \ge 3$.

\noindent 
a) (cf.\cite{zhou}) Let $H$ be a Hopf manifold which is the quotient of $\CC^n-0$ 
by the group containing an element of infinite $f \in Aut (\CC^n)$
such that $f(0)$ has eigenvalues $\xi$ such that $\vert \xi \vert <1$.
Let $E$ be a bundle on $H$ such that pullback on the universal cover 
is trivial. Then:

\begin{equation}\label{vanishingformula1}
 H^q(H,\Omega^p(E))=0 \ \ \ {\rm for} \ q\ne 0,1,n-1,n      
\end{equation}
\begin{equation}\label{vanishingformula2}
H^0(H,\Omega^p(E))=H^1(H,\Omega^p(E)) \ \ \ {\rm and} \ \ \ 
H^{n-1}(H,\Omega^p(E))=H^{n}(H,\Omega^p(E))
\end{equation}

\noindent 
b) Let $\H$ be Brieskorn van deVen manifold or a quotient 
$V-0/{\ZZ}$ where $V$ is the complete intersection (\ref{hammcase})
and $\ZZ$ is the cyclic group generated by the automorphism 
(\ref{homotety}). For a pair of finite dimensional vector spaces $A,B$, 
 denote by $\V_n(A,B)=\oplus_{i=0}^{i=n} V_i$ the graded vector space
such that $V_0=V_1=A, V_{n-1}=V_n=B$ and $V_i=0$ for $i \ne 0,1,n-1,n$,
Let $\W_{n,q}=\oplus_{i=0}^{i=n} W_i$ denote a graded vector space such 
that $\rk W_{q-1}=1, \rk W_q=2, \rk W_{q+1}=1, W_i=0 \ (i \ne q,q \pm 1)$. 
Let $E$ be a vector 
bundle on $\H$ such that its pull back on the universal cover is trivial.
If the multiplicity of $V$ at the origin is greater than one, 
then for $p \ge 1$ 
one has an isomorphism of graded spaces:
$$ \oplus_q H^q(\H,\Omega^p(E))=\oplus \V_n(H^0(\Omega^p(E),H^n(\Omega^p(E))))
\oplus \W_{n,n-p-1}$$.
\end{theorem}
Other results of this note are a calculation of the cohomology 
of local systems on homological Hopf manifolds and a study of degeneration
of Hodge-deRham spectral sequence.

My interest in this material stemmed from the lecture of Prof. X.Y.Zhou 
on Xiamen conference.
I want to thank organizers of Xiamen conference for their
hospitality and Prof. Steven Yau for providing useful references 
related to material of this paper.

\section{Generalizations of Hopf manifolds and cohomology 
of bundles on their universal covers}

\subsection{Main definitions}

{\it Notations:} Below, for a compact manifold $\H$,  
$b_i$ denotes $\rk H^i(\H,\QQ)$.

\begin{dfn} A ($\QQ$)-{\it homological 
Hopf} manifold is a compact complex manifold 
of dimension $n$ with 
$b_1=b_{2n-1}=1, b_i=0, i \ne 0,1,2n-1,2n$. 
Such a manifold $\cal H$ is an integral homological Hopf 
manifold if it has the integral cohomology 
isomorphic to $H^*(S^1 \times S^{2n-1},\ZZ)$. A $\ZZ$-homological Hopf manifold 
is called primary if its fundamental group is isomorphic to $\ZZ$.
\end{dfn}

$\QQ$-homological Hopf manifolds (which may not 
be $\ZZ$-homological), in the case $n=2$, were considered 
in \cite{ebeling}. 

\begin{dfn} A {\it topological Hopf} manifold is a complex manifold 
with universal cover being a complement to a point in a contractible
Stein space.
\end{dfn}

\begin{dfn} A {\it Hopf} manifolds is a complex manifold $\H$ for which
 the universal cover is biholomorphic to 
the complement in $\CC^n$ to the origin $O \in \CC^n$. 
A Hopf manifold is called 
primary if the Galois group of the universal covering (or equivalently 
the fundamental group of $\H$) is isomorphic to $\ZZ$.
\end{dfn}

In the last decade, many diverse constructions of 
non-Kahler manifolds were proposed (cf. \cite{messer}).
An interesting problem is to classify homological or topological 
(primary or non primary)  
Hopf manifolds. \footnote{the terminology is suggested by more studied 
problem of classification of homological projective spaces}
More precisely, one would like to describe the topological, 
differentiable and almost complex manifolds \footnote{cf. \cite{Morita}
for a discussion of invariants of almost complex structures
on Brieskorn van deVen manifolds}
which admit a complex structure yielding   
a homological or topological Hopf manifold. Moreover, one would 
like to describe the 
moduli space parametrizing the complex structures 
on such an almost complex manifold.

The main results of Kodaira on the classification of Hopf surfaces 
can be summarized as follows 
(cf. \cite{kodairanas}, \cite{kodairaamerjourn}). 
A Hopf surface is a quotient by a group $G$ which has $\ZZ$ as its center
and such that $G/{\bf Z}$ is finite (cf. \cite{kodairaamerjourn}, 
theorem 30). Any Hopf surface contains a curve.
A homological Hopf surface having algebraic dimension equal to zero and
containing at least one curve is a Hopf surface (\cite{kodairaamerjourn}, 
theorem 34; case of algebraic dimension one considered in \cite{ebeling}).

The first examples of topological Hopf manifolds, 
(which are not Hopf) were found by 
Brieskorn and Van de Ven (cf. \cite{BvdV}). The contractible Stein spaces, 
in which the complement to a point serves in \cite{BvdV} 
as the universal cover 
of the topological Hopf manifold  
are the affine hypersurfaces $V \subset \CC^{n+1}$ 
where $V$ is a zero set of weighted homogeneous 
polynomial
\begin{equation}\label{pham}
z_0^{a_0}+...+z_n^{a_n}=0
\end{equation}
Here the integers $a_i$ must satisfy conditions which 
assure that the link of a singularity (\ref{pham}) are topological 
(and possibly exotic) spheres. For example
this is the case when $n$ is odd and $a_1=3,a_2=6r-1,a_i =2$ 
(varying $r$ yields all exotic spheres bounding a parallelizable 
manifold). The examples of primary topological Hopf manifolds 
are obtained as the quotients by the action of the restriction 
of the following automorphism
of the complex linear space:

\begin{equation}\label{homotety}
  T \cdot (z_1,....)= 
(\lambda^{1 \over {a_1}}z_1,...., \lambda^{1 \over {a_i}}z_i,....)
\ \ \ (\vert \lambda \vert <1)
\end{equation}
(leaving the hypersurface (\ref{pham}) invariant).

More generally, consider a complete intersection: 

\begin{equation}\label{hammcase} 
\sum_{\nu =1}^{n+k} \alpha_{\mu,\nu}z^{a_{\nu}} \ \ \ \mu=1,...,k
\end{equation}
with generic coefficients $\alpha_{\mu,nu}$. The latter assures 
that (\ref{hammcase}) has an isolated singularity at the origin. 
Under the appropriate 
conditions (cf. \cite{hamm} Satz 1.1) the link of this singularity 
(\ref{hammcase}) is a homology sphere (over $\QQ$ or $\ZZ$). 

Note that Zaidenberg conjectures that if $V$ is set of 
zeros of a polynomial then contractible $V$ with one isolated singularity 
up to an automorphism of $\CC^n$ 
is a zero set of a weighted homogeneous polynomial (cf. \cite{Zaidenberg}).
This would imply that the non singular loci of 
hypersurfaces (\ref{pham}) are the 
only covering spaces of topological Hopf manifolds having 
an affine hypersurface as its closure.

\subsection{Topological properties of $\QQ$-Hopf manifolds}

\par \noindent 
Kodaira's result on the fundamental groups of Hopf surfaces can 
be extended to topological Hopf manifolds.

We shall start by considering 
the question when a quotient of $V-O$ is a homological Hopf manifold.

\begin{prop}\label{propfundgroups} (i) 
The fundamental group $\pi_1(\H)$ of a topological Hopf manifold
$\H$ is a central extension of $\ZZ$ by a finite group.

(ii) The class of biholomorphic equivalence of the 
universal cover is an invariant deformation type of 
a topological Hopf manifold.
Deformation type of a topological 
Hopf manifold is given by the 
type of $V-0$ and the conjugacy class of class of the 
subgroup $G \subset Aut_OV$ which 
is a central extension $\ZZ$ by a finite group.

(iii ) Let $V$ be an affine hypersurface with $\CC^*$ action 
 or a complete 
intersection (\ref{hammcase})
and 
an isolated singularity at the origin.
The quotient $V-0/\pi_1$ is a $\QQ$-homology Hopf manifold iff
 the invariant subgroup of the action of $\pi_1$ on $H^{n-1}(V-O,\QQ)^{\pi_1}$ 
is trivial. 
In particular if the link of $V$ is a $\QQ$-sphere (resp. $\ZZ$-sphere) 
then $V-O/\pi_1$ is a $\QQ$-homology (resp. integral) Hopf manifold.
\end{prop}

\begin{proof} The proof (i) is a direct generalization of Kodaira's argument. 
Since $V$ is Stein, by Remmert embedding theorem (cf. \cite{narasim})
 we can assume that $V$ is a subspace of  $\CC^N$ 
and select a ball $B \subset \CC^N$ centered at the image of $O$. 
Let $g$ be an element of 
infinite order in $\pi_1$ acting properly discontinuously on $V-O$. By Hartogs 
theorem $g$ extends to an automorphism of 
$V$ fixing $O$. Either $g$ or $g^{-1}$ 
takes $\partial B \cap V$ into $B \cap V$ and $\cap_n g^n(B \cap V)=O$ 
since existence of $z \ne O$ in the boundary of this intersection contradicts 
proper discontinuity of the action of $\pi_1$. Hence the quotient of $V-0$
by the subgroup $\{g\}$ generated by $g$ is compact. 
The index of the cyclic subgroup $\{g\}$ in $\pi_1(\H)$ is equal to the degree 
of the cover $V-0/\{g\}$ and hence is finite. The subgroup of $\{g\}$
which is normal in $\pi_1(\H)$ yields claimed presentation of the latter as 
a central extension.

(ii) is a a consequence of result of Andreotti and Vesentini on 
pseudo-rigidity of Stein manifolds (cf. \cite{andreotti}).

To see (iii) consider the Hochschild-Serre spectral sequence 
\begin{equation}\label{hoch}
E_2^{p,q}: H^p(\pi_1/\ZZ,H^q(\ZZ,W)) \rightarrow H^{p+q}(\pi_1,W)
\end{equation}
where $W$ is a $\QQ$-vector space with a structure of a $\pi_1$-module
and $\ZZ$ is the center generated by an element $g \in \pi_1$. 
Since $\pi_1/\ZZ$ is finite and $H^q(\ZZ,W)$ has no $\ZZ$-torsion we have 
$E_2^{p,q}=0$ for $p \ge 1$. Moreover $H^q(\ZZ,W)$ is the subspace of 
$g$-invariants (resp. $g$-covariants) of $W$ for $q=0$ (resp. $q=1$)
and is trivial for $q>1$. Hence in spectral sequence (\ref{hoch})
there are at most two non trivial terms and hence 
\begin{equation}\label{calculation}
\dim H^q(\pi_1,W)=
\begin{cases} W^g & {q =0} \\ W_g & {q=1} \\ 0 & q>1
\end{cases}
\end{equation}
Using a homotopy equivalence between $V-O$ and the link of isolated singularity 
of $V$ we obtain $H^q(V-O,\QQ)=0$  for $q \ne 0,n-1,n,2n-1$ 
(cf. \cite{milnor}).
Next consider the spectral sequence of the action of $\pi_1$ on the universal cover
$V-0$:
\begin{equation}
H^p(\pi_1,H^q(V-O)) \Rightarrow H^{p+q}(V-O/\pi_1)
\end{equation}
Applying 
(\ref{calculation}) for $W=\ZZ$ for $q=0,2n-1$, $W=H^{n-1}(V-O)$
or $W=H^n(V-0)$ if $q=n-1,n$ and using $\pi_1$ equivariant identification
$H^{n}(V-O)=H^{n-1}(V-O)^*$ (which is a consequence of the Poincare 
duality for the link of singularity of $V$) 
we obtain the result.   
\end{proof}

\subsection{Holomorphic automorphisms of universal covers}

Next we shall consider the question of existence of automorphisms
of $V-O$ different than (\ref{homotety}) which generate 
an infinite cyclic group acting properly discontinuously with 
compact quotient i.e. automorphisms yielding primary topological 
Hopf manifolds. We shall assume that $V$ is  
a zero set of an arbitrary weighted homogeneous polynomial 
(i.e. a sum of monomials $z_1^{i_1} \cdot \cdot \cdot z_n^{i_n}$ 
such that $\sum {i_k} b_k=d$ where $b_k$ are  the weights and 
$d$ is the degree) 
\footnote{for the hypersurface (\ref{pham}), 
one has 
$b_i={{{\rm l.c.m.}(a_1,...,a_n)} \over {a_i}}, d={\rm l.c.m.}(a_1,...,a_n)$} 
 with isolated singularities.
In the case when $V=\CC^n$ the automorphisms generating infinite
properly discontinuously acting groups with compact quotients 
were described by Haefliger (cf. \cite{haefl}). 

\begin{lemma}\label{autBvdV} Let $V$ be a zero set of a weighted homogeneous 
polynomial $f$ having weights $a_i$ and the degree $d$. 
Then one has the extension:
$$0 \rightarrow \CC^* \rightarrow Aut(V-0) \rightarrow 
G \rightarrow 0$$
where the group $G$ is finite if $\sum_i {1 \over {a_i}} < 1$. More generally, 
one has the same type exact sequence if $V$ is a complete intersection
(\ref{hammcase}) and $\sum {1 \over {a_i}} < k$.
\end{lemma}

\begin{proof} The affine hypersurface $V$ is a open subset in 
a $\PP^1$-bundle $\bar V \rightarrow V_{\infty}$ 
over a hypersurface in a weighted projective space which 
one can identify with the hyperplane section at infinity.
 $Aut V$ is the subgroup of the group $Aut^{\circ} \bar V$ of automorphisms
of $\bar V$ fixing invariant the two sections of $\bar V$ (corresponding 
to the exceptional set of the weighted blowup of projective closure of
$V$ and the hyperplane section at infinity $V_{\infty}$).
\footnote{Automorphisms of ruled surfaces were studied in \cite{maruyama}.}
Clearly, the $\CC^*$ action 
provides a normal subgroup  $\CC^*$ of $Aut^{\circ} \bar V$ and 
the quotient is the subgroup of the group of automorphisms of $V_{\infty}$.
The latter has ample canonical class if the condition on the weights
is met and hence $V_{\infty}$ has a finite automorphisms group. 
The argument in the case of complete intersection is the same.  
\end{proof}

On the other hand for the quadric hypersurface $z_1^2+....+z_n^2$ 
one has many non-trivial automorphisms (cf. \cite{totaro}). For example
one has the automorphism of 
\begin{equation}
X_1X_2+X_3^2+X_4^2+....+X_n^2
\end{equation} given by 
Danilov and Gizatullin (cf. \cite{Danilov} p.101) defined by the 
change of variables:
\begin{equation}\label{quadricauto}
X'_1=\beta X_1,\ \ \ X'_2=\beta(\alpha^2 X_2+2 \alpha X_3f(X_1)+
X_1f^2(X_1)),
\end{equation} 
$$ X'_3=\sqrt{-1}(\alpha \beta X_3+\beta X_1f(X_1)) \ \ \ 
X_i'=(\alpha \beta)X_i \ \ \ (4 \le i \le n)$$ 
preserves quadric for all $f$ and no power of (\ref{quadricauto})
has a fixed point on $\CC^n-0$. More generally, the action of the 
infinite cyclic group generated by (\ref{quadricauto}) is proper 
discontinuous.
The jacobian of (\ref{quadricauto}) at the origin 
\begin{equation} det  
\begin{pmatrix} \beta & 0 & \cr 0 & \beta \alpha^2 & 2\alpha f(0) \cr
    \beta f(0) & 0 & \alpha \beta \cr
\end{pmatrix}  
\end{equation}
Hence if eigenvalues (for $f(0)=0$, $\beta, \alpha \beta, \alpha^2 \beta$) 
have absolute value less than 1 which is 
the case for $\vert \beta \vert <1$ and $\vert \alpha \vert <1$.

\subsection{Cohomology of bundles on universal covers}

The results of sections \ref{cohovectbun} and 
\ref{hodgederham} deal primarily with the 
cohomology of bundles on topological Hopf manifolds.
They depend on vanishing of the cohomology of the bundles
on universal covers which we now review. 

In the case of topological Hopf manifolds the vanishing of the  
cohomology of bundles on the universal cover $V-O$ follows from
the vanishing of cohomology of coherent sheaves on Stein 
spaces in positive dimensions (Cartan's theorem B) 
and from the following two results:

\begin{theorem}\label{Scheja}(Scheja; cf \cite{Scheja}, \cite{Siuextension} 
p. 129) 
Let $V$ be a complex space, $A$ 
a subvariety of dimension $\le d$ and $\F$ a coherent sheaf on 
$V$ such that ${\rm codh} \F \ge d+q$. Then $H^k(V,\F) \rightarrow   
H^k(V-A,\F)$ is an isomorphism for $0 \le k < q-1$ and 
injective for $k=q-1$.
\end{theorem} 
\noindent and 

\begin{theorem}\label{siu} (Siu, cf. \cite{siuannals})
 Let $V$ be a complex space, $A$ a subvariety of dimension $\le d$, 
$i: V-A \rightarrow V$ and $\F$ and coherent analytic sheaf 
on $V-A$. If $codh \F \ge d+3$ then $i_*(\F)$ is a coherent 
sheaf (on $V$).
\end{theorem}

\begin{corollary}\label{cohomologyoncover}
 Let $V$ be a Stein space with isolated singularity $O$ and
let $\F$ be a coherent sheaf on $V-O$. Assume that  
$\dim V \ge 3$. Then $\F$ extends to a coherent sheaf on $V$ and 
$H^k(V-O,\F)=0$ for $0 <k \le d-1$ where $d$ is the cohomological 
codimension of the stalk of this extension at $O$. 
In particular if $\F$ is a locally trivial bundle on 
$V-O$ which admits a locally trivial extension to $V$
then $H^k(V-O,\F)=0$ for $k \ne 0, n-1$.
\end{corollary}

\begin{proof}
Indeed, $hd_{V-O} \F=0$ since $\F$ is locally free and hence 
$codh \F=\dim V$. Therefore $i_*(\F)$ is coherent.
Taking $X=V, A=O$ in theorem \ref{Scheja} we see 
that $H^k(V-0,\F) \rightarrow H^k(V,\pi_*(\F))$ 
is injective for $0 < k  \le d-1$ and Cartan's theorem B 
yields the first claim. In the case when extension is locally trivial
we have $d=\dim V$ and hence the second assertion.
\end{proof}

\begin{remark} Consider the case when bundle $\F$ on $\CC^n-0$ 
is a pullback of a bundle 
$\F'$ on $\PP^{n-1}$ via Hopf map $\pi: \CC^n-0 \buildrel \CC^* \over 
\rightarrow \PP^{n-1}$. Then $\F$ extends to a locally trivial bundle
on $\CC^n$ if and only if $\F'$ is a direct sum of line bundles 
(\cite{serre}).
The cohomology of $\F$ can be found using Leray spectral sequence
$H^p(\PP^{n-1},R^q\pi_*(\F)) \rightarrow H^{p+q}(\CC^n-O,\F)$.
Since by projection formula 
$R^q\pi_*(\pi^*(\F'))=R^q\pi_*(\O_{\CC^n-O})\otimes \F'$
and since the fibers of $\pi$ are Stein (in fact just $\CC^*$) 
we have $R^q\pi_*(\O_{\CC^n-O})=0$
for $q \ne 0$. Hence the above Leray sequence degenerates and 
vanishing $H^q(\CC^n-O,\F)=0$ for $0< q\le n-2$ follows from 
a well known vanishing of cohomology of line bundles on $\PP^{n-1}$ (i.e. 
in all dimensions $\ne 0,n-1$.) Note that vanishing
of cohomology of a bundle $\F'$ on $\PP^{n-1}$ 
in indicated range is closely related to 
existence of a splitting of the bundle (cf. \cite{bundles} p.39). 
\end{remark}

\begin{remark} It follows from the result of extension 
of bundles $\pi^*(\F)$ to $\CC^n$ mentioned in the last remark 
that $\pi^*(\Omega^p_{\PP^{n-1}})$ cannot be extended 
to a {\it locally trivial} bundle on $\CC^n$ and hence
corollary \ref{cohomologyoncover} does not yield information on the cohomology 
of this bundle.
\end{remark}

\begin{remark}
Note that the canonical class  
of the $\ZZ$-quotients of $\CC^{n+1}-O$ or  hypersurfaces (\ref{pham}) 
is given by effective divisor:
 \begin{equation}
  K_{\H}=\sum_{i=0}^{i={n}} D_i
 \end{equation} 
where $D_i$ divisor on $\H$ biholomorphic to Hopf (resp. Brieskorn van de Ven 
for quotients of (\ref{pham})) 
manifold which is the image of affine hypersurface in $\CC^{n+1}$ 
(resp. in (\ref{pham}))
given by $z_i=0$ 
($i=0,..n$). Indeed, the form ${{dz_0} \over {z_0}} \wedge ...\wedge 
{{dz_{n}}\over {z_n}}$ is a meromorphic form with poles
at $z_0 \cdot \cdot \cdot z_n=0$ 
on the universal cover invariant
under the deck transformations and hence descending to the quotient.
For a the quotient of the hypersurface (\ref{pham})  one has:
\begin{equation}\label{canonicalbrieskorn}
  K_{\H}=-a_0D_0+\sum_{i} D_i
\end{equation}
Indeed, the restriction to $V$ of 
invariant under the action of (\ref{homotety})
meromorphic form 
$\omega_1={{dz_1} \over {z_1}} \wedge ...\wedge 
{{dz_{n}}\over {z_n}}$ descents to a meromorphic form 
on $\H$. On the other hand the form 
$\omega_2=Res {{dz_0 \wedge .... dz_n} \over {z_0^{a_0}+...z_n^{a_n}}}=
{{dz_1 \wedge ...\wedge 
dz_{n}}\over {z_0^{a_0-1}}}$ in holomorphic and non-vanishing 
on $V-0$. The formula 
(\ref{canonicalbrieskorn}) follows by comparison of divisors 
of these forms $\omega_1$ and $\omega_2$ on $\CC^n$.
\end{remark}

In the study of the cohomology of bundles $\Omega^p(E)$ on Hopf manifolds
we restrict our-self to the case when $V-O$
is the complement to the fixed point in a 
hypersurface in $\CC^{n+1}$ which supports a $\CC^*$ action 
or complete intersection (\ref{hammcase}). 
The needed results on the cohomology of sheaves 
of differential forms are essentially contained in \cite{Yau}
in the case of hypersurfaces and in \cite{Vose} in the case 
of complete intersections.

\begin{lemma} Let $V$ be a hypersurface (\ref{pham}) or a complete
intersection (\ref{hammcase}). Then one has the following:
$${\rk} H^q(V-0,\Omega^p)=
\begin{cases}0 & p+q \le n-2 \ \ \ \  1 \le q \le n-2
\\  \tau & p+q=n-1,n \ \  1 \le q \le n-2
 \\  0  & p+q \ge n+1 \ \ \ \ 1 \le q \le n-2
\end{cases}$$
\end{lemma}

From this it follows:

\begin{lemma}\label{maxideallemma}
 If the multiplicity of $V$ at the 
origin is greater than one then for the automorphism 
$T^*_{H^q(V-0,\Omega^p)}$ 
of $H^q(V-0,\Omega^p)$ induced
by (\ref{homotety})
one has $$\dim \Ker (T^*_{H^q(V-0,\Omega^p)}-I)=1$$
\end{lemma}

\begin{proof} In \cite{Yau}, S.Yau obtained, for a hypersurface $V$, 
 an isomorphism 
$H^q(V-0,\Omega^p)$ and 
 the vector space: 
$$M_f=\CC[z_1,....,z_{n+1}]/(...{{\partial f}\over {\partial z_i}}...)$$
(or its dual). One has splitting $M_{f}=\CC \cdot 1 \oplus {\cal M}$
where $\cal M$ is the image in the quotient ring of the maximal 
ideal of the local ring at the origin. This image is a vector 
space which has the monomials $...\cdot z_i^{j_i} \cdot ...$ with 
$0 \le j_i <a_i, \sum_i j_i >0$ as its basis.
The action of automorphism $T$ is given by $z_i \rightarrow \lambda_iz_i$
and the above monomials are eigenvectors of this action with 
eigenvalues all different from 1. 
Hence $\Ker (T^*_{H^q(V-0,\Omega^p)}-I)$ corresponds to 
the summand $\CC \cdot 1$. A similar argument works 
in the case of complete intersections using Prop. 1.3 (d) of \cite{Vose}.  
\end{proof}

\section{Cohomology of local systems on topological Hopf manifolds}

\begin{prop}\label{homologylocalsystems}
 Let $\L$ be a local system on a topological Hopf manifold $\H$.
Then
$$ H^i(\L)=0 \ \ \ for \ \ i\ne 0,1,2n-1,2n $$
$$ \dim H^0(\L)=\dim H^1(\L)=\dim H^{2n-1}(\L)=\dim H^{2n}(\L)=1 $$    
\end{prop}

\begin{proof} Let $\tilde \H$ be the universal covering space.
We have the spectral sequence for the action of a subgroup $\ZZ$ 
(generated by $T$) of the 
center of $\pi_1(\H)$ on $\tilde \H$ yields:
\begin{equation} 
 H^p(\tilde \H) \buildrel T-I \over \rightarrow H^p(\tilde \H) 
\rightarrow H^p(\tilde \H/\ZZ) \rightarrow H^{p+1}(\tilde \H)
\end{equation}
Since 
$$H^p(\tilde \H,\QQ)=
\begin{cases}0 &  p \ne 0,2n-1 \\  \QQ & p=0,2n-1
\end{cases}$$
and the action of $T$ is trivial for both $p=0,2n-1$ we obtain the 
claim for the cohomology of the primary Hopf manifold $\tilde \H/\ZZ$.
The claim for the local systems on $\tilde H/\pi_1(\H)$ follows
from the spectral sequence (below $\sigma: \tilde \H/\ZZ \rightarrow \H$) 
\begin{equation}
H^p(\pi_1(\H)/\ZZ,H^q(\sigma^*(\L)) \rightarrow H^{p+q}(\H,\L)
\end{equation}
since the group $\pi_1(\H)/\ZZ$ is finite.
\end{proof}

\begin{remark} One can obtain the 
conclusion of Proposition \ref{homologylocalsystems}
for homological Hopf manifolds, if one makes additional assumptions
on the fundamental group e.g. for $\H$ such that $\pi_1(\H)$ satisfies 
the conclusions of Proposition \ref{propfundgroups} (i) or more generally
on characteristic varieties of the group $\pi_1(\H)$ (cf. \cite{libgober}). 
\end{remark}

\section{Cohomology of vector bundles}\label{cohovectbun}

Let $X$ be a complex space on which a group $G$ acts via holomorphic maps 
freely and let $\pi: X \rightarrow X/G$. Let 
$\F$ be a coherent sheaf on $X/G$. Let $\bullet ^G: V \rightarrow V^G$
be the functor assigning to a $\CC$-vector space with a $G$-action 
the subspace of invariants. One has canonical isomorphism (cf \cite{Groth}): 
$\Gamma(X,\pi^*(\F))^G=
\Gamma(X/G,\F)$. The corresponding spectral sequence of the composition 
of functors is given by (cf. \cite{Groth}): 

\begin{equation}\label{spectralsequence}
  E_2^{p,q}=H^p(G,H^q(X,\pi^*\F)) \Rightarrow H^{p+q}(X/G,\F)
\end{equation}

In particular this can be applied to the case when 
$X$ be a complex manifold and $\pi: \tilde X \rightarrow X$ is its universal 
covering space. Let $\F$ be a coherent sheaf on $X$. 
The above spectral sequence becomes:

\begin{equation}
  E_2^{p,q}= H^p(\ZZ,H^q(\tilde X,\pi^*\F)) \Rightarrow H^{p+q}(X,\F)
\end{equation}

In the case when $G=\ZZ$  and hence has the cohomological dimension equal 
to $1$, 
this spectral sequence degenerates in term $E_2$ (i.e. due 
to the vanishing $H^i(\ZZ,V)=0, i \ge 2 $ for a $\ZZ$-module $V$).
Since $E_2^{0,q}=H^q(X,\pi^*(\F))^{\ZZ}=\Ker 
(T-Id: H^q(X,\pi^*{\F}) \rightarrow H^q(X,\pi^*{\F})$ and 
$E_2^{1,q}=\Coker T-Id: H^q(X,\pi^*{\F}) \rightarrow H^q(X,\pi^*{\F})$ 

Hence we have the exact sequence: 

\begin{equation}\label{milnor}
\rightarrow H^{q-1}(X,\pi^*(\F)) \buildrel T-1 \over \rightarrow 
 H^{q-1}(X,\pi^*(\F)) \rightarrow H^q(X/G,\F) 
\end{equation}
$$\rightarrow
H^q(X,\pi^*(\F)) \buildrel T-1 \over \rightarrow 
 H^{q}(X,\pi^*(\F))$$

Note that we have $H^i(\tilde X,\O)=0$ for $i \ne 0,n-1,n$, 
if $\tilde X=\CC^n-0$

{\it Proof of the theorem \ref{cohomologyresults}}. The finiteness 
assumption and the exact sequence (for a sheaf $\F$ on $\bar V$):
\begin{equation}
  H^q(\bar V) \rightarrow H^q(V,\F) \rightarrow H^{q+1}_A(\bar V,\F)
\end{equation}
yields that $H^q(V,p^*(E))$ is finite dimensional 
for $q \le r$ and is infinite dimensional for $q=0$.
Moreover the space $H^0(V,p^*(E))$ is Montel  with 
the semi-norm given by the compact subsets $\bigcup_i ( f^i(U) \cup f^{-i}(U))$
where $U$ is the closure of the fundamental domain of an infinite
order transformation $f$ on $V$ in the center of $\pi_1(H)$. In particular 
$f^*$ is compact and hence index of $I-f^*$ acting on $H^0(V,p^*(E))$
is zero and hence has equal dimensions of kernel and cokernel. 
This and the exact sequence relating the cohomology of cover discussed above
yield the claim.
\qed

{\it Proof of the theorem \ref{vanishing}.}
We have $H^q(V-O,\Omega^p_{V-O} \otimes \pi^*(E))=0$ 
unless $q=0,n-1-p,n-p,n-1, p \ge 2$.
In the case $q=0,n-1$ spaces are infinite dimensional 
but the operators on $H^q$ induced by $T$ are Fredholm with 
zero index. 
Hence the exact sequence (\ref{milnor}) yields:

\begin{equation}
  {\rm dim} H^q(V-0/G,\Omega^p(E))={\rm dim} 
H^{q+1}(V-0/G,\Omega^p(E)) \ \ \ q=0,n-1
\end{equation}
$$H^q(V-0/G,\Omega^p(E))=0 \ \ \ q \ne n-1-p,n-p, 0,1,n-1,n $$
$${\rm dim} H^{n-2-p}(V-0/G,\Omega^p(E))=\tau, $$
$${\rm dim} H^{n-1-p}(V-0/G,\Omega^p(E))= \tau \ \ $$
 $${\rm dim} H^{n-p}(V-0/G,\Omega^p(E))=\tau \ \ p \ne 0,1$$

Indeed, for $q=0,1,n-1,n$ are cohomology of bundles on Hopf manifold
are kernel and cokernel of Fredholm operators which, as
was mentioned earlier, has zero index.
For $q \ne 0,1,n-1,n$ contribution from $H^q(V-0,\Omega^p\otimes \pi^*(E))$
contributes $\tau$ into ${\rm dim}H^q$ and ${\rm dim} H^{q+1}$ 
since the action of the covering group on the cohomology 
$H^{n-2-p}(V-0,\Omega^p(E))$ is trivial 
and 
remaining groups are zeros. 
Now the result follows from lemma \ref{maxideallemma}.

\begin{corollary} Let $\H$ be a homological Hopf manifold 
as in \ref{vanishing} 
Then $Pic(\H)=\CC^*$.
\end{corollary}

\begin{proof} We have $H^0(\H,\O)=H^1(\H,\O)=\CC$ from theorem 
\ref{vanishing}.
Hence the cohomology sequence of the exponential sequence 
$0 \rightarrow \ZZ \rightarrow \O \rightarrow \O^* \rightarrow 0$
has the form $H^1(\H,\ZZ) \rightarrow \CC \rightarrow H^1(\H,\O^*) \rightarrow 
H^2(\H,\ZZ)$. Since $H^2(\H,\QQ)=0$,
this yields the claim.
\end{proof}

\section{Hodge deRham spectral sequence}\label{hodgederham}

\bigskip

Recall that the category of local systems on 
a manifold is equivalent to the category 
of locally constant sheaves which in turn is 
equivalent to the category of locally free sheaves 
with integrable connection (cf. \cite{deligne}).
For a local system, $\L$, let 
$E_{\L}$ be the corresponding locally free 
sheaf and and $\nabla_{\L}$ be the corresponding flat connection 
on $E_{\L}$. 

\begin{prop} Let $\H$ be a Hopf manifold, $E$ be a 
bundle on $\H$ with a trivial pullback on the universal cover and 
$\L$ be the corresponding local system. Then the 
Hodge deRham spectral sequence:
\begin{equation}
H^q(\H,\Omega^p(E)) \rightarrow H^{p+q}(\H,\L)
\end{equation} 
degenerates in the term $E_2$.
\end{prop}

\begin{proof} In order to calculate the differential $d_1$ 
i.e. the map $E_1^{p,q} \rightarrow E_1^{p,q}$ for $q=0,1$
recall that the terms $E_1^{p,0}$ and $E_1^{p,1}$ are 
respectively invariants and covariants of the selfmap 
of $\Gamma(\CC^n,\Omega^p \otimes p^*(E))$ induced by a map $T: \CC^n 
\rightarrow \CC^n$ (such that $\CC^n-O/t=h$. On the 
other hand the holomorphic deRham complex on $\CC^n$:
\begin{equation}
 \Gamma(\CC^n,\Omega^p) \buildrel d \over \rightarrow 
\Gamma(\CC^n,\Omega^{p+1})
\end{equation}
is exact. The explicit construction of the holomorphic form $\eta$ such that 
for a given close form $\omega$ one has $\omega=d\eta$ shows that $\eta$ 
is $T$-invariant if $\omega$ is. Hence $d$ induced exact sequence 
of $T$-invariant forms. i.e. the term $E_2$ is zero
(in particular one can recover the vanishing results of
Prop. \ref{homologylocalsystems}).
\end{proof}

\end{document}